\documentclass[reqno,a4paper]{article}
\usepackage[english]{babel}
\usepackage{amsmath}
\usepackage{amsthm}
\usepackage{amssymb}
\usepackage{enumerate}
\usepackage{xspace}
\usepackage{euscript}
\usepackage{graphicx}
\usepackage{graphics}
\usepackage{amscd}
\usepackage{epsfig}
\usepackage{tabularx}
\usepackage[usenames]{color}

\usepackage{geometry}
\usepackage[latin1]{inputenc}
\usepackage{amsmath}
\usepackage{amsfonts}
\usepackage{amssymb}
\usepackage{graphicx}
\usepackage{epstopdf}

\newtheorem{lem}{Lemma}[section]
\newtheorem{defn}{Definition}[section]
\newtheorem{thm}[lem]{Theorem}
\newtheorem{rem}[lem]{Remark}
\newtheorem{corol}[lem]{Corollary}

\newtheorem{asm}[lem]{Assumption}
\numberwithin{equation}{section}
\newcommand{\mri}{{\mathrm{i}}} 

\title{\textbf{Time-harmonic diffuse optical tomography: H\"older stability of the derivatives of the optical properties of a medium at the boundary}}
\author{\textsc{Jason Curran\thanks{Department of Mathematics and Statistics, CONFIRM-Science Foundation Ireland, University of Limerick, Ireland, Jason.Curran$@$ul.ie},\quad Romina
Gaburro}\thanks{Department of Mathematics and Statistics, CONFIRM-Science Foundation Ireland, Health Research Institute (HRI),
University of Limerick, Ireland,
Romina.Gaburro$@$ul.ie},\\ \textsc{Clifford J. Nolan\thanks{Department of Mathematics and Statistics, CONFIRM-Science Foundation Ireland, Health Research Institute (HRI), University of Limerick, Ireland, Clifford.Nolan$@$ul.ie},\quad and \quad Erkki Somersalo\thanks{Department of Mathematics, Applied Mathematics and Statistics,
Case Western Reserve University, Cleveland, OH 44106-7058, U.S., erkki.somersalo$@$case.edu}}}

\date{}

\begin{document}
\maketitle

 \begin{abstract}
We study the inverse problem in Optical Tomography of determining the optical properties of a medium $\Omega\subset\mathbb{R}^n$, with $n\geq 3$, under the so-called \textit{diffusion approximation}. We consider the time-harmonic case where $\Omega$ is probed with an input field that is modulated with a fixed harmonic frequency $\omega=\frac{k}{c}$, where $c$ is the speed of light and $k$ is the wave number. Under suitable conditions that include a range of variability for $k$, we prove a result of H\"older stability of the derivatives of the \textit{absorption coefficient} $\mu_a$ of any order at the boundary $\partial\Omega$ in terms of the measurements, in the case when the \textit{scattering coefficient} $\mu_s$ is assumed to be known. The stability estimates rely on the construction of singular solutions of the underlying forward elliptic system, which extend results obtained in  J. Differential Equations 84 (2): 252-272 for the single elliptic equation.
\end{abstract}


\section{Introduction}\label{section introduction}
\setcounter{equation}{0}

In this paper we address the problem in diffuse Optical Tomography (OT) of determining the optical properties of a medium when light is radiated through the surface of the medium. Although Maxwell's equations provide a complete model for the light propagation in a scattering medium on a micro scale, on the scale suitable for medical diffuse OT an appropriate model is given by the \textit{radiative transfer equation} (or \textit{Boltzmann equation})\cite{ArSc}. If $\Omega$ is a domain in $\mathbb{R}^{n}$, with $n\geq 2$ with smooth boundary $\partial\Omega$ and radiation is considered in the body $\Omega$, then it is well known that if the input field is modulated with a fixed harmonic frequency $\omega$, the so-called \textit{diffusion approximation} leads to the complex partial differential equation (see \cite{Ar}) for the energy current density $u$,

\begin{equation}\label{K,q equation}
-\mbox{div}\left(K\nabla u\right)+(\mu_a -\mri k)u=0,\qquad\textnormal{in}\quad\Omega.
\end{equation}

Here $k=\frac{\omega}{c}$ is the wave number, $c$ is the speed of light and, in the anisotropic case, the so-called \textit{diffusion tensor} $K$, is the complex matrix-valued function

\begin{equation}\label{K time-harmonic 1}
K=\frac{1}{n}\Big((\mu_a-\mri k) I+(I-B)\mu_s\Big)^{-1},\qquad\textnormal{in}\quad\Omega,
\end{equation}

where $B_{ij}=B_{ji}$ is a real matrix-valued function, $I$ is the $n\times n$ identity matrix and $I-B$ is positive definite (\cite{Ar}, \cite{HS}, \cite{H}) on $\Omega$. 
The diffusion equation (\ref{K,q equation}) is obtained by projecting the Boltzmann equation to the subspace of first order spherical harmonics ($P_1$ approximation), and the matrix $B$ is the contribution of the scattering phase function. The isotropic case when $B$ and hence $K$ is a scalar multiple of the identity is obtained if the scattering phase function depends only on the angle between the incident and scattering direction (\cite{HAS}, \cite{KS}). 
The spatially dependent real-valued coefficients $\mu_a$ and $\mu_s$ are called the \textit{absorption} and the \textit{scattering coefficients} of the medium $\Omega$ respectively and represent the optical properties of $\Omega$. It is worth noticing that many tissues including parts of the brain, muscle and breast tissue have fibrous structure on a microscopic scale which results in anisotropic physical properties on a larger scale. Therefore, the model considered in this manuscript is appropriate for the case of medical applications of OT. Although it is common practise in OT to use the Robin-to-Robin map to describe the boundary measurements (see \cite{Ar}), the Dirichlet-to-Neumann (D-N) map will be employed here instead. This is justified by the fact that in OT, prescribing its inverse, the Neumnann-to-Dirichlet (N-D) map (on the appropriate spaces), is equivalent to prescribing the Robin-to-Robin boundary map. A rigorous definition of the D-N map for equation \eqref{K,q equation} will be given in section \ref{formulation problem}.



It is also well known that in the static case, where $k=0$ in \eqref{K,q equation}, \eqref{K time-harmonic 1}, prescribing the N-D map is insufficient to recover both coefficients $\mu_a$ and $\mu_s$ uniquely \cite{ArL} unless \textit{a-priori} smoothness assumptions are imposed \cite{Ha}. See also \cite{Ha2} for a further discussion on the simultaneous unique determination, in the isotropic static case, of both optical coefficients in the case when such coefficients are piecewise analytic. The static anisotropic case, for which \eqref{K,q equation} is a single real elliptic equation, was studied in \cite{G}, where the author proved Lipschitz stability of $\mu_a$ and H\"older stability of the derivatives  of $\mu_a$ at the boundary in terms of $\Lambda_{K,\mu_a}$, in the case when $\mu_s$ is assumed to be known. In the time-harmonic case, where the medium s probed with an input field which is modulated with a fixed harmonic frequency $\omega=\frac{k}{c}$, with $k\neq 0$, the forward model \eqref{K,q equation} is a complex elliptic equation.  A result of Lipschitz stability of the boundary values of $\mu_a$ in terms of the D-N map, when $\mu_s$ is again assumed known, was established in the time-harmonic anisotropic case by some of the authors in \cite{DoGLN}.

In this paper, we consider the anisotropic time-harmonic case and extend the result in \cite{DoGLN}, by stably determining the derivatives of the absorption coefficient $\mu_a$, $D^{h}\mu_a$, for any $h\geq 1$, at the boundary of an anisotropic medium $\Omega\subset\mathbb{R}^n$, $n\geq 3$, whose scattering coefficient $\mu_s$ is assumed to be known. More precisely, we show that, under suitable conditions, $D^{h}\mu_a$ at the boundary $\partial\Omega$, depends upon the D-N map of \eqref{K,q equation}, $\Lambda_{K, \mu_a}$, with a modulus of continuity of H\"older type, if $k$ is chosen in certain intervals that depend on \textit{a-priori} bounds on $\mu_a$, $\mu_s$ and on the ellipticity constant of $I-B$ (Theorem \ref{main result}). The intervals of variability for $k$ were determined in \cite{DoGLN}. 

The case where $\mu_a$ is assumed to be known and the scattering coefficient $\mu_s$ is to be determined, can be treated in a similar manner. The choice in this paper, as well as in \cite{DoGLN}, of focusing on the determination of $\mu_a$ rather than the one of $\mu_s$ is driven by the medical application of OT we have in mind. While $\mu_s$ varies from tissue to tissue, it is the absorption coefficient $\mu_a$ that carries the more interesting physiological information as it is related to the global concentrations of certain metabolites in their oxygenated and deoxygenated states \cite{Boas}.

Our main result (Theorem \ref{main result}) is based on the construction of singular solutions to the complex elliptic equation \eqref{K,q equation}, having an isolated singularity outside $\Omega$. Such solutions were first constructed in \cite{A1} for the equation

\begin{equation}\label{Calderon operator}
\mbox{div}(K\nabla u ) = 0,\qquad\textnormal{in}\quad\Omega,
\end{equation}

when $K$ is a real matrix-valued function belonging to $W^{1,p}(\Omega)$, with $p>n$. Such solutions have been employed to prove stability estimates at the boundary in \cite{A1}, \cite{AG}, \cite{AG1} and \cite{GL} in the case of Calder\'on's problem (see \cite{C}) with global, local data and on manifolds (see also  \cite{Sa} and \cite{I1}). We also recall the seminal papers \cite{A1}, \cite{Koh-V1}, \cite{Koh-V2}, \cite{N} and \cite{Sy-U} of this extensively studied companion inverse problem, together with the review papers \cite{Bo} and \cite{U}.

In this paper we extend the construction of the singular solutions of \cite{A1} to the case of elliptic equations of type \eqref{K,q equation} with complex coefficients. Such a construction is done by treating \eqref{K,q equation} as a strongly elliptic system with real coefficients, since $\Re K\geq\tilde\lambda^{-1}I>0$, where  $\tilde\lambda$ is a positive constant depending on the \textit{a-priori} information on $\mu_s$, $B$ and $\mu_a$. In \cite{DoGLN}, the authors extended the construction of singular solutions to the complex equation \eqref{K,q equation} having an isolated singularity of Green's type only. This was enough to prove the Lipschitz continuity of the boundary values of $\mu_a$ in terms of the D-N map. Here singular solutions with an isolated singularity of arbitrary high order for elliptic complex partial differential equations are constructed and employed to prove our main result of H\"older stability of the derivatives (of any order) of $\mu_a$ at the boundary in terms of the D-N map, therefore further extending the original results of \cite{A1}. 

Our result also provides a first step towards a reconstruction procedure of $\mu_a$ by boundary measurements based on a Landweber iterative method for non-linear problems studied in \cite{dHQS}, where the authors provided an analysis of the convergence of such algorithm in terms of either a H\"older or Lipschitz global stability estimates (see also \cite{A-dH-F-G-S}, \cite{B-dH-Q-S}, \cite{F-A-Ba-dH-G-S}, \cite{Fa-dH-S}). We recall the important results of \cite{A-V} and \cite{Be-Fr} of global stability estimates for Calder\'on's inverse conductivity problem in the case of real and complex isotropic conductivities, respectively, and refer to the subsequent papers \cite{A-dH-G-S}, \cite{A-dH-G-S1}, \cite{A-dH-G-S2}, \cite{Be-dH-F-S}, \cite{Be-Fr-V}, \cite{Be-dH-Fr-V-Z}, \cite{Be-dH-Q}, \cite{Be-Fr-Mo-Ro-Ve}, \cite{FGS}, \cite{G-S} for an overview of the issue of stability estimates in related inverse problems. We also refer to \cite{HebdenArridge}, \cite{KVKaAr} for further reconstruction techniques of the optical properties of a medium and to \cite{GiHA} for a topical review on diffuse OT.

The paper is organized as follows. In Section \ref{formulation problem} we rigorously formulate the problem, state the main result (Theorem \ref{main result}) of H\"older stability of the derivatives of $\mu_a$ on $\partial\Omega$ and recall a previous result of Lipschitz stability of $\mu_a$ on $\partial\Omega$ (Theorem \ref{stability of mu}) for the sake of completeness. Section \ref{section singular solutions bis} is devoted to the construction of singular solutions for the complex partial differential equation \ref{K,q equation} on a ball, having an isolated singularity (of any order) at the centre of the ball. In Section \ref{proofs main result} we give the proof Theorem \ref{main result}.


\section{Formulation of the problem}\label{formulation problem}

\subsection{Main assumptions}\label{main assumptions}
We rigorously formulate the problem by introducing the following notation, definitions and assumptions.
For $n\geq 3$, a point $x\in \mathbb{R}^n$ will be denoted by $x=(x',x_n)$, where $x'\in\mathbb{R}^{n-1}$ and $x_n\in\mathbb{R}$.
Moreover, given a point $x\in \mathbb{R}^n$, we will denote with $B_r(x), B_r'(x')$ the open balls in
$\mathbb{R}^{n},\mathbb{R}^{n-1}$, centred at $x$ and $x'$ respectively with radius $r$
and by $Q_r(x)$ the cylinder

\[Q_r(x)=B_r'(x')\times(x_n-r,x_n+r).\]

We will also denote $B_r=B_r(0)$, $B'_r=B'_r(0)$ and $Q_r=Q_r(0)$.



\begin{defn}\label{def boundary}
Let $\Omega$ be a bounded domain in $\mathbb R^n$, with $n\geq 3$. We shall say that the boundary of $\Omega$, $\partial\Omega$, is of Lipschitz class with constants $r_0,L>0$, if for any $P\in\partial\Omega$ there exists a rigid
transformation of coordinates under which we have $P=0$ and

$$\Omega\cap Q_{r_0}=\{(x',x_n)\in Q_{r_0}\: |\,x_n>\varphi(x')\},$$

where $\varphi$ is a Lipschitz function on $B'_{r_0}$ satisfying

\[\varphi(0)=0\]

and

\[\|\varphi\|_{C^{0,1}(B'_{r_0})}\leq Lr_0.\]

\end{defn}

We consider, for a fixed $k>0$,

\begin{equation}\label{L complex}
L= - \mbox{div}\left(K\nabla\cdot\right) + q,\qquad\textnormal{in}\quad\Omega,
\end{equation}

where $K$ is the complex matrix-valued function

\begin{equation}\label{K time-harmonic 2}
K(x)=\frac{1}{n}\Big((\mu_a(x)-\mri k) I+(I-B(x))\mu_s(x)\Big)^{-1},\qquad\textnormal{for\:any}\:x\in\Omega,
\end{equation}

and $q$  is the complex-valued function 

\begin{equation}\label{q}
q=\mu_a - ik\qquad\textnormal{in}\quad\Omega.
\end{equation}

Here $I$ denotes the $n\times n$ identity matrix, where the matrix $B$ is given by the OT physical experiment and it is such that $B\in L^{\infty}(\Omega, Sym_n)$, where $Sym_n$ denotes the class of $n\times n$ real-valued symmetric matrices and such that $I-B$ is a positive definite matrix (\cite{Ar}, \cite{HS}, \cite{HAS}, \cite{H}). In this paper we assume that the \textit{scattering coefficient} $\mu_s$ is also known in $\bar\Omega$ and it is the derivatives of the \textit{absorption coefficient} $\mu_a$ on $\partial\Omega$ that we seek to estimate from boundary measurements.\\

We assume that there are positive constants $\lambda$, $E$ and $\mathcal{E}$ and $p>n$ such that the known quantities $B \in L^{\infty}(\Omega, Sym_n)$, 
$\mu_s \in L^{\infty}(\Omega)$ and the unknown quantity $\mu_a \in L^{\infty}(\Omega, Sym_n)$ satisfy the two assumptions below respectively.

\begin{asm}\textnormal{(Assumption on $\mu_s$ and $B$)}\label{assumption on mus and B}\\

\begin{equation}\label{assumption on scattering coeff positive}
\lambda^{-1}\leq\mu_{s}(x)\leq \lambda,\qquad\textnormal{for\:a.e.}\: x\in\Omega,
\end{equation}

\begin{equation}\label{assumption scattering sobolev}
||\mu_s||_{W^{1,\:p}(\Omega)},\quad ||B||_{W^{1,\:p}(\Omega)} \leq E,
\end{equation}

\begin{equation}
\mathcal{E}^{-1}|\xi|^2\leq(I-B(x))\xi\cdot\xi\leq\mathcal{E}|\xi|^2,\qquad\textnormal{for\:a.e.\:}x\in\Omega,\quad\textnormal{for\:any}\:\xi\in\mathbb{R}^n,
\end{equation}


\end{asm}


\begin{asm}\textnormal{(Assumption on $\mu_a$)}\label{assumption on mua}\\

\begin{equation}\label{limitazioni per a,b}
\lambda^{-1}\leq\mu_{a}(x)\leq\lambda,\qquad\textnormal{for\:a.e.}\: x\in\Omega,
\end{equation}

\begin{equation}\label{holder a,b}
\parallel\mu_{a}\parallel_{\:W^{1,p}(\Omega)}\leq{E}.
\end{equation}



\end{asm}

We state below some facts needed in the sequel of the paper. Most of them are straightforward consequences of our assumptions.\\

The inverse of $K$,

\begin{equation}\label{K inverse}
K^{-1}= n \Big(\mu_a I + (I-B)\mu_s -ikI\Big),\qquad\textnormal{on}\quad\Omega
\end{equation}

has real and imaginary parts given by the symmetric, real matrix valued-functions on $\Omega$,

\begin{eqnarray}
K^{-1}_{R} &=& n\left(\mu_a I + (I-B)\mu_s\right),\label{K inverse 1}\\
K^{-1}_I &=& \!\!\!\!-nkI,\label{K inverse 2}
\end{eqnarray}

respectively. As an immediate consequence of assumptions \ref{assumption on mus and B}, \ref{assumption on mua} we have

\begin{eqnarray}\label{apriori assumptions inequalities}
n\lambda^{-1}(1+\mathcal{E}^{-1})|\xi|^2\leq K^{-1}_R (x) \xi\cdot\xi &\leq & n\lambda(1+\mathcal{E})|\xi|^2,\label{apriori ass ineq 1}\\
- K^{-1}_I (x)\xi \cdot \xi &=& nk|\xi|^2,\label{apriori ass ineq 2}
\end{eqnarray}

for a.e. $x\in\Omega$ and any $\xi\in\mathbb{R}^n$. Moreover $K^{-1}_{R}$ and $K^{-1}_I$  commute, therefore the real and imaginary parts of $K$ are the symmetric, real matrix valued-functions on $\Omega$,

\begin{eqnarray}
K_R &=& \frac{1}{n}\bigg(\Big( \mu_a I + (I-B)\mu_s \Big)^2 + k^2I\bigg)^{-1}\big(\mu_a I + (I-B)\mu_s \big)\label{K1},\\
K_I  &=& \frac{k}{n}\bigg(\Big( \mu_a I + (I-B)\mu_s \Big)^2 + k^2I\bigg)^{-1},\label{K2}
\end{eqnarray}

respectively. Assumptions \ref{assumption on mus and B}, \ref{assumption on mua} also imply that  

\begin{eqnarray}
& & K_R(x) \xi\cdot\xi \geq\frac{\lambda (1+\mathcal{E})}{n} \Big(\lambda^2 (1+\mathcal{E})^2 + k^2\Big)^{-1} |\xi|^2 ,\label{K1 pos def}\\
& & K_I(x) \xi\cdot\xi \geq\frac{k}{n} \Big(\lambda^{2} (1+\mathcal{E})^2 + k^2 \Big)^{-1} |\xi|^2 \label{K2 neg def},
\end{eqnarray}

for a.e. $x\in\Omega$, for every $\xi\in\mathbb{R}^n$ and the \textit{boundness condition}

\begin{equation}\label{boundness OT}
|K_R(x)|^2 + |K_I(x)|^2\leq \Big(\lambda^{-2} (1+\mathcal{E}^{-1})^2 + k^2\Big)^{-2}\:\Big(\frac{\lambda^2 (1+\mathcal{E})^2 + k^2}{n^2}\Big), 
\end{equation}

for a.e. $x\in\Omega$.

Moreover $K=\{K^{hk}\}_{h,k=1,\dots ,n}$ and $q$ satisfy

\begin{equation}\label{boundness K}
||K^{hk}||_{W^{1,p}(\Omega)}\leq C_1,\qquad h,k=1,\dots , n
\end{equation}

and 

\begin{equation}\label{boundness q}
|q(x)| = |\mu_a (x) -ik|\leq \lambda +k,\qquad\textnormal{for\:a.e.}\:x\in\Omega,
\end{equation}

respectively, where $C_1$ is a positive constant depending on $\lambda$, $E$, $\mathcal{E}$, $k$ and $n$.


By denoting $q=q_R+iq_I$, the complex equation

\begin{equation}\label{complex eq}
-\mbox{div}\left(K\nabla u\right)+qu=0,\qquad\textnormal{in}\:\Omega
\end{equation}

is equivalent to the system for the vector field $u=(u^1, u^2)$

\begin{equation}\label{system OT}
\left\{ \begin{array}{ll} - \textnormal{div}(K_R\nabla
u^1) +  \textnormal{div}(K_I\nabla
u^2) + \left(q_Ru^1-q_Iu^2\right)=0, &
\textrm{$\textnormal{in}\quad\Omega$},\\
-  \textnormal{div}(K_I\nabla
u^1) -  \textnormal{div}(K_R\nabla
u^2 ) + \left(q_Iu^1+q_R u^2\right)=0,&
\textrm{$\textnormal{in}\quad\Omega$},
\end{array} \right.
\end{equation}

which can be written in a more compact form as

\begin{equation}\label{system compact}
-\mbox{div}(C\nabla u) + qu = 0,\qquad\textnormal{in}\quad\Omega
\end{equation}

or, in components, as

\begin{equation}\label{system compact components}
-\frac{\partial}{\partial x_h}\left\{C_{lj}^{hk} \frac{\partial}{\partial x_k} u^j \right\} + q_{lj}u^j = 0,\qquad\textnormal{for}\quad l=1,2, \quad\textnormal{in}\quad\Omega,
\end{equation}

where $\left\{C_{lj}^{hk}\right\}_{h,k=1,\dots, n}$ is defined by

\begin{equation}\label{def C}
C_{lj}^{hk}=K_R^{hk}\delta_{lj}-K_I^{hk}\left(\delta_{l1}\delta_{j2} - \delta_{l2}\delta_{j1}\right),
\end{equation}
or, in block matrix notation,
\begin{equation}
C = \left(\begin{array}{rr} K_R & - K_I \\ K_I & K_R\end{array}\right) \in \mathbb{R}^{2n\times 2n},
\end{equation}

and $\{q_{lj}\}_{l,j=1,2}$ is a $2\times 2$ real matrix valued function on $\Omega$ defined by

\begin{equation}\label{def q}
q_{lj}=q_R\delta_{lj}-q_I\left(\delta_{l1}\delta_{j2} - \delta_{l2}\delta_{j1}\right),
\end{equation}

or,
\begin{equation}\label{matrix q}
q = \left(\begin{array}{rr} q_R & - q_I \\ q_I & q_R\end{array}\right) =  \left(\begin{array}{rr} \mu_a  &  k \\ -k & \mu_a\end{array}\right) 
 \in \mathbb{R}^{2\times 2}.
\end{equation}

Observing that for $\xi = (\xi_1,\xi_2)\in\mathbb{R}^{2n}$, using the symmetry of $K_I$, we have
\begin{equation}
 C\xi\cdot \xi = K_R\xi_1\cdot \xi_1 + K_R\xi_2\cdot \xi_2,
\end{equation}
the estimate
\eqref{K1 pos def}, together with \eqref{boundness OT} imply that system \eqref{system OT} is \textit{uniformly elliptic} and \textit{bounded}, therefore it satisfies the \textit{strong ellipticity condition}

\begin{equation}\label{strong ellipticity}
C_2^{-1}|\xi|^2\leq  C\xi\cdot \xi \leq C_2 |\xi|^2, \qquad\textnormal{for\:a.e.}\: x\in\Omega ,\quad \textnormal{for\:all}\:\xi  \in\mathbb{R}^{2n},
\end{equation}

where $C_2 >0$ is a constant depending on $\lambda$, $\mathcal{E}$, $k$ and $n$.

\begin{rem}
The matrix  $q$ given by (\ref{matrix q}) 
is uniformly positive definite on $\Omega$ and it satisfies

\begin{equation}\label{q positive definite}
\lambda^{-1}|\xi|^2\leq q \xi\cdot\xi = \mu_a|\xi|^2 \leq\lambda |\xi|^2,\qquad\textnormal{for\:a.e.}\: x\in\Omega,\quad\textnormal{for\:every}\:\xi\in\mathbb{R}^2.
\end{equation}

\end{rem}

\begin{defn}\label{a priori data}
We will refer in the sequel to the set of positive numbers $r_0$, $L$, $\lambda$, $E$, $\mathcal{E}$ introduced above, along with
the space dimension $n$, $p>n$, the wave number $k$ and the diameter of $\Omega$, $diam(\Omega)$, as to the \textit{a-priori data}. 
\end{defn}


\subsection{The Dirichlet-to-Neumann map}\label{D-to-N}
Let $K$ be the complex matrix-valued function on $\Omega$ introduced in \eqref{K time-harmonic 2} and $q=\mu_a - ik$,  satisfying assumptions \ref{assumption on mus and B}, \ref{assumption on mua}. $B$ and $\mu_s$ are assumed to be known in $\Omega$ and satisfying assumption \ref{assumption on mus and B}, so that $K$ is completely determined by $\mu_a$, satisfying assumption \ref{assumption on mua}, on $\Omega$. Denoting by $\langle\cdot,\cdot\rangle$ the $L^{2}(\partial\Omega)$-pairing between $H^{\frac{1}{2}}(\partial\Omega)$ and its dual $H^{-\frac{1}{2}}(\partial\Omega)$, we will emphasise such dependence of $K$ on $\mu_a$ by denoting $K = K_{\mu_a}$.

For any $v,w\in\mathbb{C}^n$, with $v=(v_1,\dots , v_n)$, $w=(w_1,\dots , w_n)$, we will denote throughout this paper by $v\cdot w$, the expression
\[v\cdot w = \sum_{i=1}^{n} v_i w_i.\]

\begin{defn}
The Dirichlet-to-Neumann (D-N) map corresponding to $\mu_a$ is the operator

\begin{equation}\label{mappaDN}
\Lambda_{\mu_a}:H^{\frac{1}{2}}(\partial\Omega)\longrightarrow{H}^{-\frac{1}{2}}(\partial\Omega)
\end{equation}

defined by

\begin{equation}\label{def DN}
\langle\Lambda_{\mu_a}\:f,\:\overline{g}\rangle\:=\:\int_{\:\Omega}\Big( K_{\mu_a}(x) \nabla{u}(x)\cdot\nabla\varphi(x)+(\mu_a(x)-ik)u(x)\varphi(x)\Big)\:dx,
\end{equation}

for any $f$, $g\in H^{\frac{1}{2}}(\partial\Omega)$, where $u\in{H}^{1}(\Omega)$ is the weak solution of

\begin{displaymath}
\left\{ \begin{array}{ll} -\textnormal{div}(K_{\mu_a}(x)\nabla
u(x))+ (\mu_a-ik)(x)u(x)=0, &
\textrm{$\textnormal{in}\quad\Omega$},\\
u=f, & \textrm{$\textnormal{on}\quad{\partial\Omega}$}
\end{array} \right.
\end{displaymath}

and $\varphi\in H^{1}(\Omega)$ is any function such that $\varphi\vert_{\partial\Omega}=g$ in the trace sense.
\end{defn}


Given $B$, $\mu_s$, $\mu_{a_i}$, and the corresponding diffusion tensors $K_{\mu_{a_i}}$, for $i=1,2$, satisfying assumptions \ref{assumption on mus and B}, \ref{assumption on mua}, the well known Alessandrini's identity (see \cite[(5.0.4), p.129]{A1})

\begin{eqnarray}\label{Alessandrini identity}
\langle \left(\Lambda_{\mu_{a_1}} - \Lambda_{\mu_{a_2}}\right) f,\overline{g}\rangle &=& \int_{\Omega} \left(K_{\mu_{a_1}}(x) - K_{\mu_{a_2}}(x)\right)\nabla u(x)\cdot\nabla v(x)\:dx\nonumber\\
&+&\int_{\Omega}\left(\mu_{a_1}(x)-\mu_{a_2}(x)\right)u(x)v(x)\:dx,
\end{eqnarray}

holds true for any $f, g\in H^{\frac{1}{2}}(\partial\Omega)$, where $u, v\in H^{1}(\Omega)$ are the unique weak solutions to the Dirichlet problems

\begin{displaymath}
\left\{ \begin{array}{ll} -\textnormal{div}(K_{\mu_{a_1}}(x)\nabla
u(x)) +  (\mu_{a_1}-ik)u(x)=0, &
\textrm{$\textnormal{in}\quad\Omega$},\\
u=f, & \textrm{$\textnormal{on}\quad{\partial\Omega}$}
\end{array} \right.
\end{displaymath}
and

\begin{displaymath}
\left\{ \begin{array}{ll} -\textnormal{div}(K_{\mu_{a_2}}(x)\nabla
v(x)) +  (\mu_{a_2}-ik)v(x)=0, &
\textrm{$\textnormal{in}\quad\Omega$},\\
v=g, & \textrm{$\textnormal{on}\quad{\partial\Omega}$},
\end{array} \right.
\end{displaymath}

respectively.\\

We will denote in the sequel by $\parallel\cdot\parallel_{\mathcal{L}(H^{\frac{1}{2}}(\partial\Omega),H^{-\frac{1}{2}}(\partial\Omega))}$ the norm on the Banach space of bounded linear operators between $H^{\frac{1}{2}}(\partial\Omega)$ and $H^{-\frac{1}{2}}(\partial\Omega)$.


\subsection{Main result}\label{sub main result}

%
%

Before stating the main result (Theorem \ref{main result}), for the sake of completeness we recall the Lipschitz stability estimate at the boundary for $\mu_a$ in \cite{DoGLN}.

\begin{thm}\label{stability of mu}(\textnormal{Lipschitz stability of boundary
values}). Let $n\geq 3$, and $\Omega$ be a bounded domain in $\mathbb{R}^n$ with
Lipschitz boundary with constants $L, r_0$ as in definition \ref{def boundary}. If $p>n$, $B$, $\mu_{s}$ and $\mu_{a_{i}}$, for $i=1,2$, satisfy assumptions \ref{assumption on mus and B}, \ref{assumption on mua} and the wave number $k$ satisfies either

\begin{equation}\label{k range 1}
0<k \leq k_0 :=\frac{ \sqrt{\lambda^2 (1+\mathcal{E})^2 + \lambda^{-2} (1+\mathcal{E}^{-1})^{2} \tan^{2}\left(\frac{\pi}{2n}\right)} -\lambda (1+\mathcal{E})}{\tan\left(\frac{\pi}{2n}\right)},
\end{equation}

or

\begin{equation}\label{k range 2}
k \geq \tilde{k}_0 :=\frac{1 +  \sqrt{1 + \tan^{2}\left(\frac{\pi}{2n}\right)}}{\tan\left(\frac{\pi}{2n}\right)}\: \lambda (1+\mathcal{E}),
\end{equation}

where, $\lambda$ and $\mathcal{E}$ are the positive numbers introduced in assumptions \ref{assumption on mus and B}, \ref{assumption on mua}, then 

\begin{equation}\label{stabilita' anisotropa}
\parallel\mu_{a_{1}}-\mu_{a_{2}}\parallel_{L^{\infty}\:(\partial\Omega)}
\leq{C}\parallel\Lambda_{\mu_{a_1}}-\Lambda_{{\mu_{a_2}}}\parallel_{\mathcal{L}(H^{\frac{1}{2}}(\partial\Omega),H^{-\frac{1}{2}}(\partial\Omega))},
\end{equation}

\noindent where $C>0$ is a constant depending on $n$, $p$, $L$, $r_0$, $diam(\Omega)$, $\lambda$, $E$, $\mathcal{E}$ and $k$.
\end{thm}

Here we address the issue of stably determining the derivatives of any order $h\in\mathbb{N}$ of $\mu_a$ at the boundary in terms of the D-N map. We assume that $\textnormal{supp}(B)$ is compactly contained in $\Omega$, i.e. that the material is isotropic while not necessarily homogenous near the boundary $\partial\Omega$, which is a reasonable assumption in many imaging applications.

\begin{thm}\label{main result}(\textnormal{H\"older stability of boundary
derivatives}). Let $n\geq 3$, and $\Omega$ be a bounded domain in $\mathbb{R}^n$ with
Lipschitz boundary with constants $L, r_0$ as in definition \ref{def boundary}. Let $p>n$, $B$, $\mu_{s}$ and $\mu_{a_{i}}$, for $i=1,2$, satisfy assumptions \ref{assumption on mus and B}, \ref{assumption on mua}, with $\textnormal{supp}(B)\Subset\Omega$. Assume that 

\begin{equation}
|| (\mu_{a_{1}}-\mu_{a_{2}})||_{C^{h,\alpha} (\overline\Omega_{r})} \leq E_{h},\label{smoothness assumption mua}
\end{equation}

for some integer $h \geq 1$, where $\Omega_r=\left\{x\in\overline\Omega\:|\:\textnormal{dist}(x,\partial\Omega)<r\right\}$ and that the wave number $k$ satisfies either \eqref{k range 1} or \eqref{k range 2}, where $\lambda$ and $\mathcal{E}$ are the positive numbers introduced in assumptions \ref{assumption on mus and B}, \ref{assumption on mua}. Then

\begin{equation}\label{stability derivatives}
\parallel D^{h}(\mu_{a_{1}}-\mu_{a_{2}})\parallel_{L^{\infty}\:(\partial\Omega)}
\leq{C}\parallel\Lambda_{\mu_{a_1}}-\Lambda_{{\mu_{a_2}}}\parallel^{\delta_{h}}_{\mathcal{L}(H^{\frac{1}{2}}(\partial\Omega),H^{-\frac{1}{2}}(\partial\Omega))},
\end{equation}

where $\delta_{h} = \Pi_{i=0}^{h} \frac{\alpha}{\alpha + i}$ and $C>0$ is a constant depending on $n$, $p$, $L$, $r_0$, $diam(\Omega)$, $\lambda$, $E$, $\mathcal{E}$, $h$ and $k$.

\end{thm}

A H\"older stability estimate at the boundary of the derivatives of the diffusion tensor $K_{\mu_a}$ in terms of $\Lambda_{\mu_a}$ follows as a straightforward consequence of Theorem \ref{main result} under more stringent regularity conditions.

\begin{corol}\label{corollary}
Suppose the hypotheses of Theorem \ref{main result} are satisfied. Moreover, assume that
\begin{eqnarray}
|| \mu_{a_{i}}||_{C^{h,\alpha} (\overline\Omega_{r})} &\leq & E_{h},\qquad\textnormal{for}\quad i=1,2,\label{smoothness assumption mua single}\\
||\mu_s||_{C^{h,\alpha} (\overline\Omega_{r})} &\leq & E_{h},\label{assumption ms smooth}\\
||B||_{C^{h,\alpha} (\overline\Omega_{r})} &\leq & E_{h},\label{assumption B smooth}
\end{eqnarray}

then

\begin{equation}\label{stability derivatives K}
\parallel D^{h}(K_{\mu_{a_{1}}} - K_{\mu_{a_{2}}})\parallel_{L^{\infty}\:(\partial\Omega)}
\leq{C}\parallel\Lambda_{\mu_{a_1}}-\Lambda_{{\mu_{a_2}}}\parallel^{\delta_{h}\alpha}_{\mathcal{L}(H^{\frac{1}{2}}(\partial\Omega),H^{-\frac{1}{2}}(\partial\Omega))},
\end{equation}

where $h$, $\delta_{h}$, $\alpha$, $\Omega_r$ are as in Theorem \ref{main result} and $C>0$ is a constant depending on $n$, $p$, $L$, $r_0$, $diam(\Omega)$, $\lambda$, $E$, $\mathcal{E}$, $h$ and $k$.

\end{corol}

\section{Singular solutions}\label{section singular solutions bis}

We consider 

\begin{equation}\label{L complex 2}
L= - \mbox{div}\left(K\nabla\cdot\right) +q,\qquad\textnormal{in}\quad B_R(z)=\Big\{x\in\mathbb{R}^n\:  \big | \: |x-z|<R\Big\},
\end{equation}

where $K=\{K^{hk}\}_{h,k=1,\dots , n}$ and $q$ are the complex matrix valued-function and the complex function respectively introduced in section \ref{section introduction} and satisfying assumptions \ref{assumption on mus and B}, \ref{assumption on mua} on $B_R(z)$.\\

\begin{thm}\textnormal{(Singular solutions)}.\label{theor singular sol}
Given $L$ on $B_R(z)$ as in \eqref{L complex 2} with $B(z)=0$, for any $m=0,1,2,\dots$, there exists $u\in{W}_{loc}^{2,\:p}(B_{R}(z)\setminus\{z\})$ such that

\begin{equation}\label{solution with potential}
Lu=0,\qquad in\quad B_R(z)\setminus\{z\},
\end{equation}

with

\begin{equation}\label{sing solution G}
u(x) \!=\!\left(K^{-1}(z)(x-z)\cdot (x-z)\right)^{\frac{2-n-m}{2}}\!\!\!\!\!\!m! \left(K_{nn}^{-1}(z)\right)^{\frac{m}{2}}\!\!C_m^{\frac{n-2}{2}}\!\!\left(\frac{K_{(n)}^{-1}(z)(x-z)}{\left(K_{nn}^{-1}(z) K^{-1}(z)(x-z)\cdot (x-z)\right)^{\frac{1}{2}}}\right) \!+ \!w(x),
\end{equation}

where $C_m^{\frac{n-2}{2}}:\mathbb{C}\rightarrow\mathbb{C}$ is the complex Gegenbauer polynomial of degree $m$ and order $\frac{n-2}{2}$ (see \cite{E}) defined on  $|z|\leq 1$ and $K_{(n)}^{-1}(z)$, $K_{nn}^{-1}(z)$ denotes the last row, last entry in the last row of matrix  $K^{-1}(z)$, respectively. Moreover, $w$ satisfies





\begin{equation}\label{w stima lipschitz}
\vert{\:w}(x)\vert+\vert{\:x-z}\:\vert\:\vert{D}w(x)\vert\leq{C}\:\vert{\:x-z}\:\vert^
{\:2-n+\alpha},\quad{in}\quad{B}_{R}(z)\setminus\{z\},
\end{equation}

\begin{equation}\label{w stima int}
\bigg( \int_{r<\vert{x-z}\vert<2r} \vert{D}^{2}w\vert^{p}\bigg) ^
{\frac{1}{p}}\leq{C}\:r^{\frac{n}{p}-n+\alpha},\quad\mbox{for}\:every\quad{r}
, \:0<r<R/2.
\end{equation}



Here $\alpha$ is such that $0<\alpha<1-\frac{n}{p}$, and C is a positive constant depending only on $\alpha,\:n,\:p,\:R$, $\lambda$, $E$, $\mathcal{E}$ and $k$.
\end{thm}


\begin{rem}\label{principal branch}
Since $K^{-1}(z)$ is a complex matrix, the expression 

\begin{equation}\label{u singular principal part}
\bigg(K^{-1}(z)(x-z)\cdot (x-z)\bigg)^{\frac{1}{2}} 
\end{equation}

appearing in \eqref{sing solution G} is defined as the principal branch of \eqref{u singular principal part}, where a branch cut along the negative real axis of the complex plane has been defined for $z^{\frac{1}{2}}$, $z\in\mathbb{C}$. Expressions like \eqref{u singular principal part} will appear in the sequel of the paper and they will be understood in the same way.
\end{rem}

\begin{rem}\label{remark sing sol simplified}
Under assumption $B(z)=0$ the singular solution in \eqref{sing solution G} simplifies to

\begin{equation}\label{sing sol simplified}
u(x) = m! \left(\mu_a(z) + \mu_s(z) - ik\right)^{\frac{2-n}{2}}\: |x-z|^{2-n-m}C_m^{\frac{n-2}{2}}\left(\frac{(x-z)_n}{|x-z|}\right) + w(x),
\end{equation}

where $w$ satisfies \eqref{w stima lipschitz} and \eqref{w stima int}.

\end{rem}

Next we consider three technical lemmas that are needed for the proof of Theorem \ref{theor singular sol}. We set $z=0$ to simply our notation.


\begin{lem}\label{tech lemma 1}
Let $p>n$ and $u\in W^{2,p}_{loc}(B_R \setminus\{0\})$ be a complex-valued function such that, for some positive $s$,

\begin{eqnarray}
& &|u(x)|\leq |x|^{2-s},\qquad\mbox{for any}\quad x\in B_R
\setminus\{0\},\\
& & \left(\int_{r<|x|<2r} |Lu|^p\right)^{\frac{1}{p}}\leq A
r^{\frac{n}{p}-s},\qquad\mbox{for\:any}\:\;r,\quad 0<r<\frac{R}{2}.
\end{eqnarray}

Then we have

\begin{eqnarray}
& & |Du(x)|\leq C |x|^{1-s},\quad\mbox{for any}\quad x\in
B_R\setminus\{0\},\label{estimate Du}\\
& &\left(\int_{r<|x|<2r} |D^2 u|^p\right)^{\frac{1}{p}}\leq C
r^{\frac{n}{p}-s}\quad\mbox{for any}\:\;r,\quad 0< r <
\frac{R}{4},\label{estimate D2u}
\end{eqnarray}
where $C$ is a positive constant depending only on $A$, $n$, $p$, $\lambda$, $E$, $\mathcal{E}$ and $k$.
\end{lem}

\begin{proof}[Proof of Lemma \ref{tech lemma 1}] For a complete proof of this lemma we refer to [Lemma 3.3., \cite{DoGLN}], which is based on on the interior $L^{p}$ - Schauder estimate for uniformly elliptic systems

\begin{equation}\label{Schauder Lp}
\Big(\int_{r<|x|<2r} |D^2 u|^p\Big)^{\frac{1}{p}} \leq C \bigg\{\Big(\int_{\frac{r}{2}<|x|<4r} |L u|^p\Big)^{\frac{1}{p}}
 +  r^{-2} \Big( \int_{\frac{r}{2}<|x|<4r}|u|^p \Big)^{\frac{1}{p}}\bigg\},
\end{equation}



for every $r$, $0<r<\frac{R}{4}$, where $C>0$ depends on $n$, $p$, $\lambda$, $E$, $\mathcal{E}$ and $k$ only.

\end{proof}






 


\begin{lem}\label{tech lemma 2}
Let $f\in L^{p}_{loc}(B_R \setminus\{0\})$ be a complex-valued function that satisfies

\begin{equation}\label{estimate f}
\left(\int_{r<|x|<2r} |f|^p\right)^{\frac{1}{p}}\leq A
r^{\frac{n}{p}-s},\qquad\mbox{for any}\:r,\quad 0<r<\frac{R}{2},
\end{equation}

with $2<s<n<p$. Then there exists $u\in W^{2,p}_{loc}(B_R\setminus\{0\})$ satisfying

\begin{equation}\label{TLu}
Lu=f,\qquad\mbox{in}\quad B_R \setminus\{0\}
\end{equation}

and

\begin{equation}\label{estimate u}
|u(x)|\leq C |x|^{2-s},\quad\mbox{for any}\quad x\in
B_R\setminus\{0\},
\end{equation}

where $C$ is a positive constant depending only on $A$, $s$, $n$, $p$, $R$, $\lambda$, $E$, $\mathcal{E}$ and $k$.
\end{lem}

\begin{proof}[Proof of Lemma \ref{tech lemma 2}] See [Lemma 3.4 \cite{DoGLN}] for a proof of this lemma.

\end{proof}

Let $L_0$ be the operator defined by 

\begin{equation}\label{operator L0}
L_0:=- \mbox{div}\left(K(0)\nabla\cdot\right),\qquad\textnormal{in}\quad B_R.
\end{equation}

Note that if we assume that $B(0)=0$, then

\begin{equation}\label{KSuppB}
K(0) =  \frac{1}{n(\mu_a (0) + \mu_s (0) - ik)} I
\end{equation}

and $L_0$ simplifies to

\begin{equation}\label{L0SuppB}
L_0 = \frac{1}{n(\mu_a (0) + \mu_s (0) - ik)} \Delta.
\end{equation}

\begin{lem}\label{tech lemma 3}
Let $f\in L^{p}_{loc}(B_R \setminus\{0\})$ be a complex-valued function that satisfies

\begin{equation}\label{estimate f}
\left(\int_{r<|x|<2r} |f|^p\right)^{\frac{1}{p}}\leq A
r^{\frac{n}{p}-s},\qquad\mbox{for any}\:\;r,\quad 0<r<\frac{R}{2},
\end{equation}

where $s>n$ is a non-integral real number and $A$ a positive constant. Then there exists $u\in W^{2,p}_{loc}(B_R\setminus\{0\})$ satisfying

\begin{equation}\label{TLu}
\Delta u=f,\qquad\mbox{in}\quad B_R \setminus\{0\}
\end{equation}

and

\begin{equation}\label{estimate u}
|u(x)|\leq C |x|^{2-s},\quad\mbox{for any}\quad x\in
B_R\setminus\{0\},
\end{equation}

where $C$ is a positive constant depending only on $A$, $s$, $n$, $p$, $R$, $\lambda$, $E$, $\mathcal{E}$ and $k$.
\end{lem}

\begin{proof}
The proof of this lemma is a straightforward adaptation of the argument of \cite[Proof of Lemma 2.3]{A1} to the case of $f$ complex-valued function.  We provide the proof for the sake of completeness. 

Letting $\Gamma (x-y) = - C_n \vert x - y \vert^{2-n}$ be the fundamental solution for the Laplace operator in $\mathbb{R}^n$, we have that for $|y|<|x|$

\begin{equation}\label{Gamma expansion}
\Gamma (x-y) = -  C_n\sum_{j=0}^{\infty} \frac{\vert y \vert^j}{\vert x \vert^{j+n-2}}C_j^{\frac{n-2}{2}}\Big( \frac{x}{\vert x \vert} \cdot \frac{y}{\vert y \vert}\Big),
\end{equation}

where $C_j^{\frac{n-2}{2}}$ are the Gegenbauer polynomials and $|C_j^{\frac{n-2}{2}}(t)|\leq \textnormal{Const}\: j^{n-3}$, for $|t|\leq 1$. Defining for $\nu = 0, 1, 2, \dots $

\begin{equation}\label{Gamma expansion truncated}
\Gamma_{\nu}(x-y) = \Gamma(x-y) -  C_n\sum_{j=0}^{\nu} \frac{\vert y \vert^j}{\vert x \vert^{j+n-2}}C_j^{\frac{n-2}{2}}\Big( \frac{x}{\vert x \vert} \cdot \frac{y}{\vert y \vert}\Big),
\end{equation}

we have that 
 
\begin{equation}
\Delta_x \Gamma_{\nu} (x, y) = \delta(x - y),\qquad\textnormal{for}\quad x\neq 0. 
\end{equation}

Assuming without loss of generality (by means of a cut-off argument) that $f \in L^{\infty} (B_R)$ and setting
 
\begin{equation}
u(x) = \int_{B_R} \Gamma_{\nu}(x,y) f(y) dy,
\end{equation}

with $\nu = [s] -n$, we have that

\begin{equation}
\Delta u = f,\qquad\textnormal{in}\quad B_R\setminus\{0\}.
\end{equation}

To prove \eqref{estimate u}, setting $P_j(x,y):= \frac{|y|^j}{|x|^{j+n-2}}\:C_{j}^{\frac{n-2}{2}}\left(\frac{x}{|x|}\cdot \frac{y}{|y|}\right)$, we write

\begin{eqnarray}\label{integral 1}
\int_{B_R} \Gamma_{\nu}(x-y) f(y) dy &=& \int_{\left\{\frac{|x|}{2}<|y|<R\right\}} \Gamma(x-y) f(y) dy\nonumber\\
&-& \sum_{j=0}^{\nu} \int_{\left\{\frac{|x|}{2}<|y|<R\right\}} P_j(x,y) f(y) dy\nonumber\\
&+&  \sum_{j=\nu +1}^{\infty} \int_{\left\{|y|<\frac{|x|}{2}\right\}} P_j(x,y) f(y) dy.
\end{eqnarray}

The first integral on the right hand side of \eqref{integral 1} can be estimated as follows,

\begin{equation}\label{int 1}
\left|\int_{\left\{\frac{|x|}{2}<|y|<R\right\}} \Gamma(x-y) f(y) dy\right|\leq C \int_{\left\{\frac{|x|}{2}<|y|<R\right\}} |x-y|^{2-n} \: |f(y)| dy \leq C|x|^{2-s}.
\end{equation}

To estimate the second and third integrals on the right hand side of \eqref{integral 1}, we will make use of inequality

\begin{equation}\label{holder inequality}
\int_{\left\{r<|y|<2r\right\}} |y|^{\mu}\:|f(y)| dy \leq C(\mu,s,n,A) \int_{\left\{r<|y|<2r\right\}} |y|^{\mu -s} dy,
\end{equation}

which holds true for any $f$ complex-valued function satisfying \eqref{estimate f}, hence the second integral on the right hand side of \eqref{integral 1} can be estimated by

\begin{equation}\label{integral 2}
\left| \sum_{j=0}^{\nu} \int_{\left\{\frac{|x|}{2}<|y|<R\right\}} P_j(x,y) f(y) dy\right|  \leq C \sum_{j=0}^{\nu} j^{n-3} \int_{\left\{\frac{|x|}{2}<|y|<R\right\}} \frac{|y|^j}{|x|^{j+n-2}} |f(y)| dy,
\end{equation}

where $C$ is a positive constant depending only on $n$, $\lambda$, $\mathcal{E}$, $k$ and $s$. Extending $f$ outside $B_R$ by setting $f=0$ on $\mathbb{R}^n\setminus B_R$ and noting that $j-s+n<0,\textnormal{for}\:j=0,\dots,\nu,$

we have

\begin{eqnarray}\label{integral 3}
\left| \sum_{j=0}^{\nu} \int_{\left\{\frac{|x|}{2}<|y|<R\right\}} P_j(x,y) f(y) dy\right| & \leq & C  \sum_{j=0}^{\nu} j^{n-3} |x|^{2-n-j}\int_{\left\{|y|>\frac{|x|}{2}\right\}} |y|^{j-s} dy\nonumber\\
& \leq & C |x|^{2-s},
\end{eqnarray}

where $C$ is a positive constant depending only on $n$, $\lambda$, $\mathcal{E}$, $k$ and $s$. Similarly, for the third integral on the right hand side of \eqref{integral 1}, noticing that $j-s+n>0,\textnormal{for}\:j>\nu,$

we have

\begin{eqnarray}
\left| \sum_{j=\nu +1}^{\infty} \int_{\left\{|y|<\frac{|x|}{2}\right\}} P_j(x,y) f(y) dy \right| &\leq & C \sum_{j=\nu +1}^{\infty}  j^{n-3} |x|^{2-n-j} \int_{\left\{|y|<\frac{|x|}{2}\right\}} |y|^{j-s} dy\nonumber\\
&\leq & C |x|^{2-s} \sum_{j=\nu +1}^{\infty}  j^{n-4}2^{-j}\leq C |x|^{2-s},
\end{eqnarray}

where $C$ is a positive constant depending only on $n$, $\lambda$, $\mathcal{E}$, $k$ and $s$.
\end{proof}


Next we proceed with the proof of Theorem \ref{theor singular sol}. 

\begin{proof}[Proof of Theorem \ref{theor singular sol}] We start by considering the fundamental solution to

\begin{equation}\label{L K(z)}
L_{K(z)}:=\mbox{div}\left(K(z)\nabla\cdot\right)
\end{equation}

in $\mathbb{R}^n$ with pole at $y$,

\[u_0(x):=\Gamma_{K(z)}(x-y)=\Big(K^{-1}(z) (x-y)\cdot (x-y)\Big)^{\frac{2-n}{2}},\qquad\textnormal{in}\quad B_R(z)\setminus\{z\}.\]

By an induction argument, we have that for $m=0,1,2,\dots$,

\begin{eqnarray}\label{partial derivative H}
u_m := \frac{\partial ^{m} u_0}{\partial y_n ^m} \bigg\vert_{y=z} &= &\bigg(K^{-1}(z)(x-z)\cdot (x-z)\bigg)^{\frac{2-n-m}{2}}\nonumber\\
& \times & \sum_{j=0}^{\lfloor\frac{m}{2}\rfloor} \left(\prod_{k=0}^{m-j-1} \left(\frac{2-n}{2} - k\right)\right) (2 K_{nn}^{-1}(z))\nonumber\\
& \times & \frac{m!}{2^j(m-2j)! j!} \frac{\left(-2K_{(n)}^{-1}(z) (x-z)\right)^{m-2j}}{\bigg(K^{-1}(z)(x-z)\cdot (x-z)\bigg)^{\frac{m-2j}{2}}}
\end{eqnarray}

which leads to

\begin{equation}\label{um1}
u_m = \bigg(K^{-1}(z)(x-z)\cdot (x-z)\bigg)^{\frac{2-n-m}{2}}m! \left(K_{nn}^{-1}(z)\right)^{\frac{m}{2}} C_{m}^{\frac{n-2}{2}}\left[\frac{K_{(n)}^{-1}(z) (x-z)}{\Big(K_{nn}^{-1}(z)K^{-1}(z)(x-z)\cdot (x-z)\Big)^{\frac{1}{2}}}\right]
\end{equation}



where


\begin{equation}\label{Gegenbauer complex}
C_{m}^{\frac{n-2}{2}}(\widetilde{z}) =  \sum_{j=0}^{\lfloor\frac{m}{2}\rfloor}  (-1)^j \frac{\Gamma\left(m-j + \frac{n-2}{2}\right)}{\Gamma\left(\frac{n-2}{2}\right) j! (m-2j)!} (2\widetilde{z})^{m-2j}
\end{equation}

is the Gegenbauer polynomial of order $m$ and degree $\frac{n-2}{2}$ in the complex variable $\widetilde{z}$.




Observing that for any $m=1,2,\dots$

\begin{equation}\label{partial derivative H solution}
L_{K(z)} u_m = 0, \qquad\textnormal{in}\quad B_R(z)\setminus\{z\},
\end{equation}

we aim to find a solution $w$ to

\begin{equation}\label{solution Lq}
Lw = - Lu_m, \qquad\textnormal{in}\quad B_R(z)\setminus\{z\},
\end{equation}

satisfying \eqref{w stima lipschitz}, \eqref{w stima int}, where $L$ is defined by \eqref{L complex}. We have

\begin{equation}\label{Lq H}
-Lu_m  = \Big(K_{ij} (x) - K_{ij}(z)\Big)\:\frac{\partial^{2}u_m}{\partial x_i \partial
x_j} + \frac{\partial K_{ij}}{\partial x_i}\:\frac{\partial
u_m}{\partial x_j} - qu_m
\end{equation}

and recalling that for any $i,j =1,\dots , n$, $K_{ij}\in W^{1,p}(B_R (z))$, with $p>n$, we have that $K_{ij}$ is H\"older continuous on $\overline{B_R (z)}$ with H\"older coefficient $\beta= 1 -\frac{n}{p}$ and hence

\begin{eqnarray}
\bigg(\int_{r<|x-z|<2r} |L u_m| ^p \bigg)^{\frac{1}{p}} 
& \leq & \bigg(\int_{r<|x-z|<2r}
|x-z|^{\beta p}\:|x-z|^{-(n+m)p}\bigg)^{\frac{1}{p}}\nonumber\\
& + & \bigg(\int_{r<|x-z|<2r} \left|\frac{\partial K_{ij}}{\partial
x_i}\right|^p\:|x-z|^{(1-n-m)p}\bigg)^{\frac{1}{p}}\nonumber\\
& + & \bigg(\lambda\int_{r<|x-z|<2r} |x-z|^{(2-n-m)p} \bigg)^{\frac{1}{p}} \nonumber\\
& \leq & C r^{\frac{n}{p} -(n+m-\beta)},
\end{eqnarray}

where $C$ is a positive constant depending on $n$, $p$, $R$, $\lambda$, $E$, $\mathcal{E}$ and $k$. Since $n+m-\beta>n$, by Lemma \ref{tech lemma 3} there is $w_0\in W^{2,p}_{loc}(B_R(z)\setminus\{z\})$ solution to $L_{K(z)}w_0= -Lu_m$, satisfying

\begin{equation}\label{estimate w0}
|w_0(x)|\leq C |x-z|^{2-(n+m-\beta)},\quad\textnormal{for\: any}\:x\in B_R(z)\setminus\{z\}.
\end{equation}

Notice that

\[L w_0 = (L-L_{K(z)}) w_0 - Lu_m,\]

and that by combining \eqref{estimate w0} together with Lemma \ref{tech lemma 1}, we obtain

\begin{eqnarray}
\bigg(\int_{r<|x-z|<2r} |(L_{K(z)} - L)w_0| ^p \bigg)^{\frac{1}{p}}  
& \leq & \bigg(\int_{r<|x-z|<2r}
|x-y|^{\beta p}\:|x-z|^{-(n+m)p}\bigg)^{\frac{1}{p}}\nonumber\\
& + & \bigg(\int_{r<|x-z|<2r} \left|\frac{\partial K_{ij}}{\partial
x_i}\right|^p\:|x-y|^{(1-n-m)p}\bigg)^{\frac{1}{p}}\nonumber\\
& + & \bigg(\lambda\int_{r<|x-z|<2r} |x-y|^{(2-n-m)p} \bigg)^{\frac{1}{p}} \nonumber\\
& \leq & C r^{\frac{n}{p} -(n+m-2\beta)}.
\end{eqnarray}

Letting $\alpha$ be an irrational number, with $0<\alpha<\beta$ and setting $J=\lfloor \frac{m}{\alpha}\rfloor$, note that $n+m -j \alpha >n$, for $j = 1,\dots ,J$; $n+m-(J+1)\alpha < n$ and $n+m-(J+1)\alpha$ is a non-integral real number. Defining, for $j=1,\dots , J-1$, $w_j$ to be the solution to

\[L_{K(z)} w_j = (L_{K(z)} - L)w_{j-1}\qquad\textnormal{in}\quad B_R(z)\setminus\{z\},\]

given by Lemma \ref{tech lemma 3}, we have

\begin{eqnarray}
& &\!\!\!\!\!\!\!\!\!\!\! |w_j (x)|\leq C |x-z| ^{2-n-m + (j+1)\alpha},\quad\textnormal{for\:any}\:x,\: x\in B_R(z)\setminus\{z\},\label{estimate wj}\\
& &\!\!\!\!\!\!\!\!\!\!\!\!\!\!\!\!\bigg(\!\!\int_{r<|x-z|<2r} \!\!\!\!\!\!\!\!\!\!\!\!\!\!|(L_{K(z)} - L)w_j| ^p \bigg)^{\frac{1}{p}}\!\!\!\leq C r^{\frac{n}{p} -n-m + (j+2)\alpha},\quad\textnormal{for\:any}\:r,\: 0<r<\frac{R}{2}.\label{estimate Lp wj}
\end{eqnarray}

For $j=J$ define $W_J$ to be the solution to

\[L W_J = (L_{K(z)} - L)w_{J-1}\qquad\textnormal{in}\quad B_R(z)\setminus\{z\}\]

given by Lemma \ref{tech lemma 2}, which satisfies

\begin{equation}\label{estimate WJ}
|W_J(x)|\leq C |x-z|^{2-n-m+(J+1)\alpha},\quad\textnormal{for\: any}\:x\in B_R(z)\setminus\{z\}.
\end{equation}

Finally, forming

\begin{equation}\label{w}
w = \sum_{j=0}^{J-1} w_j + W_J,
\end{equation}

we have that

\begin{equation}\label{w solution}
Lw = \sum_{j=0}^{J-1} (L - L_{K(z)}) w_j + \sum_{j=0}^{J-1} L_{K(z)} w_{j} + (L_{K(z)} - L) w_{J-1}  = - Lu_m
\end{equation}

and

\begin{equation}\label{w estimate}
|w(x)| \leq C \sum_{j=0}^{J} |x-y| ^{2-n-m +(j+1)\alpha}\leq C |x-y|^{2-n-m+\alpha},\qquad\textnormal{in}\quad B_R(z)\setminus\{z\},
\end{equation}

where $C>0$ is a constant depending on $n$, $p$, $R$, $\lambda$, $E$, $\mathcal{E}$ and $k$.  Properties \eqref{w stima lipschitz}, \eqref{w stima int} follow from Lemma \ref{tech lemma 1}. 

\end{proof}

We shall also need the following lemma. We set again $z=0$ to simply the notation.

\begin{lem}\label{lemma grad sing sol}
Let the hypotheses of Theorem \ref{theor singular sol} be satisfied. Then, for any $m=0,1,2,\dots$, the singular solution $u$ constructed in Theorem \ref{theor singular sol} on $B_R$ and having an isolated singularity at $z=0$ also satisfies

\begin{equation}\label{sing sol est 2}
|Du(x)|> C |x|^{1-(n+m)},\quad\mbox{for every}\: x\in\Omega,\qquad 0<|x|<r_{0},
\end{equation}

where $C>0$ and $r_0>0$ are constants depending on $\lambda$, $\mathcal{E}$, $E$, $m$, $\Omega$, $n$, $p$ and $k$.
\end{lem}

\begin{proof}
Recall that, for $m=0,1,2,\dots$, the singular solution of Theorem \ref{theor singular sol} with singularity at $z=0$ is given by

\begin{equation}\label{sing sol lemma}
u(x) = C \left(K^{-1}(0)x\cdot x\right)^{\frac{2-n-m}{2}}\: C_m^{\frac{n-2}{2}}\left(\frac{K_{(n)}^{-1}(0)x}{\left(K_{nn}^{-1}(0) K^{-1}(0)x\cdot x\right)^{\frac{1}{2}}}\right) + w(x),
\end{equation}

where $C= m! \left[(K^{-1}(0)_{nn}\right]^{\frac{m}{2}}$, $K_{(n)}^{-1}(0)$ and  $K_{nn}^{-1}(0)$ denote the last row and the last entry in the last row of matrix 

\begin{equation}\label{K inverse at 0}
K^{-1}(0) = n \Big(\mu_a(0) I + (I-B(0))\mu_s(0) - ik I\Big),
\end{equation}

$C_m^{\frac{n-2}{2}}(\widetilde{z})$ is the Gegenbauer polynomial on the complex plane and $w$ satisfies \eqref{w stima lipschitz} and \eqref{w stima int}. Recalling that under the assumption that $B(0) = 0$, \eqref{sing sol lemma} simplifies to (see lemma \ref{remark sing sol simplified})

\begin{equation}\label{sing sol simplified}
u(x) = m! \left(\mu_a(z) + \mu_s(z) - ik\right)^{\frac{2-n}{2}}\: |x|^{2-n-m}C_m^{\frac{n-2}{2}}\left(\frac{x_n}{|x|}\right) + w(x),
\end{equation}

to prove \eqref{sing sol est 2}, it is enough to show that

\begin{equation}\label{norm squared}
\left | D\left[ m! \left(\mu_a(z) + \mu_s(z) - ik\right)^{\frac{2-n}{2}}\: |x|^{2-n-m} C_m^{\frac{n-2}{2}}\left(\frac{x_n}{|x|}\right) \right]\right |^2 >C |x|^{2 - 2(n+m)},
\end{equation}

where $C>0$ is a constant depending on $\lambda$, $\mathcal{E}$, $E$, $k$, $\Omega$, $n$ and $m$. Setting $\frac{x_n}{|x|}:=t$ and recalling  \eqref{assumption on scattering coeff positive}, \eqref{limitazioni per a,b}, we obtain




\begin{eqnarray}\label{norm squared 2}
& & \left | D\left[ m! \left(\mu_a(z) + \mu_s(z) - ik\right)^{\frac{2-n}{2}}\: |x|^{2-n-m} C_m^{\frac{n-2}{2}}\left(\frac{x_n}{|x|}\right) \right]\right |^2\nonumber\\
&\geq & C |x|^{2-2(n+m)}\left\{ (2-n-m)^2 \left(C_m^{\frac{n-2}{2}}\left(t\right)\right)^2 + \left(\frac{d C_m^{\frac{n-2}{2}}}{dt} (t)\right)^2 (1-t^2) \right\}.
\end{eqnarray}

Recalling as in \cite[Proof of Lemma 3.1]{A1} that $C_m^{\frac{n-2}{2}}\left(\pm 1\right)\neq 0$ and that $C_m^{\frac{n-2}{2}}(t)$ solves 

$$(t^2 -1)w'' + 2t (n-1) w' - m(m+n-2) w = 0,$$

it follows by the Cauchy uniqueness theorem that $C_m^{\frac{n-2}{2}}(t)$ and $\frac{dC_m^{\frac{n-2}{2}}}{dt} (t)$ cannot vanish simultaneously for any $t$, $|t|\leq1$, which concludes the proof.

\end{proof}



\section{Proof of the main result}\label{proofs main result}

Since the boundary $\partial\Omega$ is Lipschitz, the normal unit vector field might not be defined on $\partial\Omega$. We shall therefore introduce a unitary vector field $\widetilde\nu$ locally defined near $\partial\Omega$ such that: (i) $\widetilde\nu$ is $C^{\infty}$ smooth, (ii) $\widetilde\nu$ is non-tangential to $\partial\Omega$ and it points to the exterior of $\Omega$
(see \cite[Lemmas 3.1-3.3]{AG} for a precise construction of $\widetilde\nu$). Here we simply recall that any point $z_{\tau}=x^{0}+\tau\widetilde\nu$, where $x^{0}\in\partial\Omega$, satisfies

\begin{equation}\label{def tau0}
C\:\tau\leq{d}(z_{\tau},\:\partial\Omega)\leq\tau,
\quad\textnormal{for}\:\textnormal{any}\quad\tau,\quad
0\leq\tau\leq\tau_{0},
\end{equation}

where $\tau_{0}$ and $C$ depend on $L$, $r_0$ only.\\

Several constants depending on the \textit{a-priori data} introduced in Definition \ref{a priori data} will appear in the proof of the main result below. In order to simplify our notation, we shall denote by $C$ any of these constants, avoiding in most cases to point out their specific dependence on the \textit{a-priori data} which may vary from case to case.\\

To simplify our notation in what follows, we will denote $K_{\mu_{a_i}}$, $K^{-1}_{\mu_{a_i}}$ and $\Lambda_{\mu_{a_i}}$ simply by $K_i$, $K_i^{-1}$ and $\Lambda_i$ respectively, for $i=1,2$. We will also denote the operator norm $\vert\vert \Lambda_1 - \Lambda_2 \vert\vert_{L(H^{1/2}(\partial \Omega), H^{-1/2}(\partial \Omega))}$ simply by $ \vert\vert \Lambda_1 - \Lambda_2 \vert \vert_*$.

\begin{proof}[Proof of Theorem \ref{main result}]

We start by recalling that by Alessandrini's Identity, \eqref{Alessandrini identity}

\begin{eqnarray*}\label{Alessandrini identity II}
\langle (\Lambda_1 - \Lambda_2 )u_1,\overline{u}_2\rangle &=& \int_{\Omega} \left(K_1 (x) - K_2 (x)\right)\nabla u_1(x)\cdot\nabla u_2(x)\:dx\nonumber\\
&+&\int_{\Omega}\left(\mu_{a_1}(x)-\mu_{a_2}(x)\right)u_1(x)u_2(x)\:dx,
\end{eqnarray*}

for any $u_1,u_2 \in H^1 (\Omega)$ that solve

\begin{eqnarray}
& & \mbox{div}\big(K_1 \nabla u_1\big) + (\mu_{a_1} -ik)u_1 = 0,\qquad\textnormal{in}\quad\Omega ,\label{eq1}\\
& & \mbox{div}\big(K_2 \nabla u_2\big) + (\mu_{a_2} -ik)u_2 = 0,\qquad\textnormal{in}\quad\Omega .\label{eq2}
\end{eqnarray}

We set $x^{0}\in\partial\Omega$ such that

\begin{equation*}
(-1)^{h}\frac{\partial^h}{\partial\tilde\nu^h}(\mu_{a_1}-\mu_{a_2})(x^{0})\:=\:\left\|\frac{\partial^h}{\partial\tilde\nu^h}\left(\mu_{a_1}-\mu_{a_2}\right)\right\|_{L^{\infty}(\partial\Omega)}\
\end{equation*}

and $z_{\tau}=x^{0}+\tau\widetilde\nu$, with $0<\tau\leq\tau_{0}$, where $\tau_{0}$ is the number fixed in \eqref{def tau0}. Let $m>0$ an integer and $u_1,\:u_2\in W^{2,p}(\Omega)$ be the singular solutions of Theorem \ref{theor singular sol} to \eqref{eq1}, \eqref{eq2} respectively, having a singularity at $z_{\tau}$ 

\begin{eqnarray}\label{singular ui}
u_1(x) & =& C_1\Big(K^{-1}_1(z_{\tau})(x-z_{\tau})\cdot(x-z_{\tau})\Big)^{\frac{2-n-m}{2}}\:C_{m,1}^{\frac{n-2}{2}}(x,z_{\tau})+ O\left(\left|x-z_{\tau}\right|^{2-n-m+\alpha}\right),\nonumber\\
u_2(x) &=& C_2 \Big(K^{-1}_2(z_{\tau})(x-z_{\tau})\cdot(x-z_{\tau})\Big)^{\frac{2-n-m}{2}}\:C_{m,2}^{\frac{n-2}{2}}(x,z_{\tau}) + O\left(\left|x-z_{\tau}\right|^{2-n-m+\alpha}\right),
\end{eqnarray}

where $C_i = m! \left[(K_i^{-1}(z_{\tau}))_{nn}\right]^{\frac{m}{2}}$, for $i=1,2$ and

\begin{equation}\label{Cm}
C_{m,i}^{\frac{n-2}{2}}(x,z_{\tau}) := C_m^{\frac{n-2}{2}}\left(\frac{(K_{i}^{-1})_{(n)}(z_{\tau})}{\left((K_{i}^{-1}(z_{\tau}))_{nn}K_i^{-1}(z_{\tau})(x-z_{\tau})\cdot(x-z_{\tau})\right)^{\frac{1}{2}}}\right),\qquad\textnormal{for}\quad i = 1,2.
\end{equation}

To prove \eqref{stability derivatives}, we will prove

\begin{equation}\label{InductProof}
\Bigg\vert \Bigg\vert \frac{\partial^j}{\partial \tilde{\nu}^j} (\mu_{a_1} - \mu_{a_2}) \Bigg\vert \Bigg\vert_{L^{\infty}(\partial \Omega)} \leq C \vert \vert \Lambda_{\mu_{a_1}} - \Lambda_{\mu_{a_2}} \vert \vert_*^{\delta_j}, \qquad\textnormal{for\:any}\:j,\quad j \leq h,
\end{equation}

by induction on $j$, where $\delta_j = \prod_{i=0}^{j} \frac{\alpha}{\alpha + i}$. \eqref{stability derivatives} will follow by \eqref{InductProof} combined with an iterative use of the interpolation inequality

\begin{equation}\label{interpolation inequality}
||Df||_{L^{\infty}(\partial\Omega)}\leq C(\Omega)\left\{ \bigg|\bigg| \frac{\partial f}{\partial\tilde\nu}\bigg|\bigg|_{L^{\infty}(\partial\Omega)} + ||f||^{\frac{\alpha}{1+\alpha}}_{L^{\infty}(\partial\Omega)}\: + ||f||^{\frac{\alpha}{1+\alpha}}_{C^{1+\alpha}(\overline\Omega)} \right\},
\end{equation}

which holds true for any $f\in C^{1+\alpha}(\overline\Omega)$, $0<\alpha\leq 1$ (see, for instance, (3.3) in \cite[Lemma 3.2]{A1}). 

For $j=0$, \eqref{InductProof} is given by \eqref{stabilita' anisotropa}.  Assuming \eqref{InductProof} holds true for any $j$, $j\leq h-1$, we will prove it holds true for $j=h$ too. By setting $\rho=2\tau_{0}$ we have $\Omega \cap B_{\rho}(z_{\tau}) \neq\emptyset$ and from \eqref{Alessandrini identity II} and we obtain


\begin{eqnarray}\label{AlesIdnI}
\vert\vert\Lambda_1 - \Lambda_2 \vert\vert_{*} \||u_1||_{H^{\frac{1}{2}}(\partial\Omega)}\:||\overline{u}_2||_{H^{\frac{1}{2}}(\partial\Omega)} & \geq & \Bigg\vert \int_{\Omega \cap B_{\rho}(z_{\tau})} (K_1(x) - K_2(x)) \nabla u_1(x) \cdot \nabla u_2(x) \: dx \Bigg\vert\nonumber\\
& - & \int_{\Omega \backslash B_{\rho}(z_{\tau})} \vert K_1(x) - K_2(x) \vert \vert\nabla u_1(x) \vert \vert \nabla u_2(x) \vert \: dx \nonumber\\
& - & \int_{\Omega \cap B_{\rho}(z_{\tau})} \vert (\mu_{a_1} - \mu_{a_2})(x) \vert \vert u_1(x)\vert  \vert u_2(x) \vert  \: dx \nonumber\\
& - & \int_{\Omega \backslash B_{\rho}(z_{\tau})} \vert (\mu_{a_1} - \mu_{a_2})(x) \vert \vert u_1(x) \vert \vert u_2(x) \vert \: dx.
\end{eqnarray}

Recalling that $K_i(x) := K(x, \mu_{a_i}) = \frac{1}{n}\Big((\mu_{a_i}-ik)I+(I-B)\mu_s\Big)^{-1}$, where it is understood that $B=B(x)$, $\mu_s=\mu_s(x)$ and $\mu_{a_i}= \mu_{a_i}(x)$ for $i=1,2$, then by Lagrange Theorem we have for any $x\in\Omega\cap B_{\rho}(z_{\tau})$
\begin{equation}\label{Lagrange}
K_1(x) - K_2(x) = \frac{\partial K(x, t)}{\partial t}\Bigg\vert_{t=c(x)} (\mu_{a_1}(x) - \mu_{a_2}(x))
\end{equation}
with $\mu_{a_1}(x) < c(x) < \mu_{a_2}(x)$ and 
\begin{equation}\label{K(x,t)}
K(x, t) = \frac{1}{n}\Big((t-ik)I+(I-B)\mu_s\Big)^{-1}.
\end{equation}

Observing that $\frac{\partial K (x, t)}{\partial t} = -nK^2(x,t)$, and that the integrals over $\Omega\setminus B_{\rho}(z_{\tau})$ on the right hand side of \eqref{AlesIdnI} are bounded by a constant $C>0$ that depends on the \textit{a-priori} data, we obtain from \eqref{AlesIdnI}

\begin{eqnarray}\label{AlesIdjI}
\vert\vert\Lambda_1 - \Lambda_2 \vert\vert_{*} \vert\vert u_1 \vert\vert_{H^{\frac{1}{2}}(\partial\Omega)} \vert \vert u_2\vert \vert_{H^{\frac{1}{2}}(\partial\Omega)} &\geq &\Bigg\vert \int_{\Omega \cap B_{\rho}(z_{\tau})} (\mu_{a_1} - \mu_{a_2})(x)nK^2(x, c(x)) \nabla u_1(x) \cdot \nabla u_2(x) \: dx \Bigg\vert \nonumber\\
& - & C -  \int_{\Omega \cap B_{\rho}(z_{\tau})} \vert (\mu_{a_1} - \mu_{a_2})(x) \vert \vert u_1(x)\vert  \vert u_2(x) \vert \: dx.
\end{eqnarray}

Next, to estimate from below the first integral on the right hand side of \eqref{AlesIdjI}, we show that

\begin{equation}\label{KGradUEst}
 \vert K^2(x, c(x)) \nabla u_1(x) \cdot \nabla u_2(x) \vert  \geq  C \vert x - z_{\tau} \vert^{2-2(n+m)},\qquad\textnormal{for\:any}\:x,\quad x\in\Omega\cap B_{\rho}(z_{\tau}).
\end{equation}

To prove \eqref{KGradUEst} we observe that \eqref{singular ui}, \eqref{w stima lipschitz} lead to

\begin{eqnarray}\label{GradUMinusGradURes}
\vert \nabla u_1 - \nabla u_2 \vert  &\leq &C\Big\{ \vert x - z_{\tau} \vert^{1-n-m}\big| \mu_{a_1}(z_{\tau}) - \mu_{a_2}(z_{\tau})\big| +  \vert x - z_{\tau} \vert^{1-n-m+\alpha}\Big\}\nonumber\\
&\leq &C\Big\{ \vert x - z_{\tau} \vert^{1-n-m}\big| \mu_{a_1}(x^0) - \mu_{a_2}(x^0)\big| +   \vert x - z_{\tau} \vert^{1-n-m}\tau^{\beta} + \vert x - z_{\tau} \vert^{1-n-m+\alpha}\Big\}\nonumber\\
\end{eqnarray}

and recalling that $\vert x - z_{\tau} \vert > C\tau$, for any $x\in\Omega\cap B_{\rho}(z_{\tau})$ and $\alpha < \beta$,

\begin{equation}\label{GradUMinusGradUResult}
\vert \nabla u_1 - \nabla u_2 \vert \leq C \Big\{\vert x - z_{\tau} \vert^{1-n-m}\vert \mu_{a_1}(x_0) - \mu_{a_2}(x_0) \vert + \vert x - z_{\tau} \vert^{1-n-m+\alpha}\Big\}.
\end{equation}

Hence, for almost every $x \in B_{\rho}(z_{\tau}) \cap \Omega$, we have

\begin{eqnarray}\label{inequality 1}
|K^2(x, c(x)) \nabla u_1 \cdot \nabla u_2| & \geq & |K^2(x, c(x)) \nabla u_1 \cdot \nabla u_1| - |K^2(x, c(x))|\:|\nabla u_1|\: |\nabla u_2  - \nabla u_1| \nonumber \\
&\geq& C \vert  x - z_{\tau} \vert^{2-2(n+m)}\nonumber\\
&-& C \vert x - z_{\tau} \vert^{1-n-m} \bigg\{ \vert x - z_{\tau} \vert^{1-n-m}\vert \mu_{a_1}(x_0) - \mu_{a_2}(x_0) \vert + C \vert x - z_{\tau} \vert^{1-n-m+\alpha} \bigg\} \nonumber \\
&=& C \vert  x - z_{\tau} \vert^{2-2(n+m)}\bigg\{ 1 - \vert \mu_{a_1}(x_0) - \mu_{a_2}(x_0) \vert - \vert x - z_{\tau} \vert^{\alpha} \bigg\}.
\end{eqnarray}

Recalling that  \eqref{InductProof} for $j=0$ leads to

\begin{equation}
\vert \mu_{a_1}(x_0) - \mu_{a_2}(x_0) \vert \leq C \vert\vert \Lambda_{\mu_{a_1}} - \Lambda_{\mu_{a_2}} \vert\vert_*  
\end{equation}

and that without loss of generality we can assume that 

\begin{equation}\label{small measurements}
\vert\vert \Lambda_{\mu_{a_1}} - \Lambda_{\mu_{a_2}} \vert\vert_* \leq \frac{1}{2C^2},
\end{equation}

we obtain from \eqref{inequality 1} that

\begin{equation}
|K^2(x, c(x)) \nabla u_1 \cdot \nabla u_2| \geq C \vert x - z_{\tau} \vert^{2-2(n+m)}\bigg\{\frac{1}{2} - \vert x - z_{\tau} \vert^{\alpha} \bigg\}
\end{equation}

for almost every $x \in B_{\rho}(z_{\tau}) \cap \Omega$ and, by possibly reducing $\rho$ so that $\vert x - z_{\tau} \vert < \frac{1}{4}$, \eqref{KGradUEst} follows. Hence, without loss of generality we can assume that

\begin{equation}\label{estimate real part}
\Re \{K^2(x, c) \nabla u_1 \cdot \nabla u_2 \} \geq   C \vert x - z_{\tau} \vert^{2-2(n+m)}.
\end{equation}

for almost every $x \in B_{\rho}(z_{\tau}) \cap \Omega$. Hence \eqref{estimate real part}, combined together with \eqref{AlesIdjI}, leads to

\begin{eqnarray}\label{AlesIdjII}
\vert\vert\Lambda_1 - \Lambda_2 \vert\vert_{*} \vert\vert u_1 \vert\vert_{H^{\frac{1}{2}}(\partial\Omega)} \vert \vert u_2\vert \vert_{H^{\frac{1}{2}}(\partial\Omega)} &\geq &\Re \Bigg\{ \int_{\Omega \cap B_{\rho}(z_{\tau})} (\mu_{a_1} - \mu_{a_2})(x)nK^2(x, c(x)) \nabla u_1 \cdot \nabla u_2 \: dx \Bigg\}\nonumber \\
& - & C -  \int_{\Omega \cap B_{\rho}(z_{\tau})} \vert (\mu_{a_1} - \mu_{a_2})(x) \vert \vert u_1(x)\vert  \vert u_2(x) \vert \: dx,
\end{eqnarray}

therefore

\begin{eqnarray}\label{Ireland}
\vert\vert\Lambda_1 - \Lambda_2 \vert\vert_{*} \vert\vert u_1 \vert\vert_{H^{\frac{1}{2}}(\partial\Omega)} \vert\vert u_2\vert\vert_{H^{\frac{1}{2}}(\partial\Omega)} & \geq & C \int_{\Omega \cap B_{\rho}(z_{\tau})} ( \mu_{a_1} - \mu_{a_2} )(x) \vert x - z_{\tau} \vert^{2-2(n+m)} \: dx \nonumber \\
& - & C - \int_{\Omega \cap B_{\rho}(z_{\tau})}\!\!\!\! \vert (\mu_{a_1} - \mu_{a_2})(x) \vert \vert u_1(x)\vert  \vert u_2(x) \vert \: dx,
\end{eqnarray}

where, without loss of generality we assumed that $(\mu_{a_1} - \mu_{a_2})(x)>0$, for almost every $x \in B_{\rho}(z_{\tau}) \cap \Omega$. Noticing that any $x \in \Omega \cap B_{\rho}(z_{\tau})$ can uniquely be written as $x = y - s \tilde{\nu}$, with  $y \in \partial \Omega$, then for any $x \in B_{\rho}(z_{\tau}) \cap \Omega$ and by using Taylor's theorem, we have 

\begin{equation}\label{mu1minusmu2TaylorEstResult}
 \sum_{j=0}^{h-1} \frac{\partial^j}{\partial \tilde{\nu}^j}(\mu_{a_1} - \mu_{a_2})(y) \frac{(-s)^j}{j!} + \frac{\partial^h}{\partial \tilde{\nu}^h}(\mu_{a_1} - \mu_{a_2})(x_0) \frac{(-s)^h}{h!} - (\mu_{a_1} - \mu_{a_2})(x) \leq C \vert x - x_0 \vert^{\alpha} s^h,
\end{equation}

for any $x \in B_{\rho}(z_{\tau}) \cap \Omega$ and

\begin{eqnarray}\label{mu1minusmu2TaylorEstResult2}  
(\mu_{a_1} - \mu_{a_2})(x) &\geq &  \Bigg\vert\Bigg\vert  \frac{\partial^h}{\partial \tilde{\nu}^h}(\mu_{a_1} - \mu_{a_2})   \Bigg\vert\Bigg\vert_{L^{\infty}(\partial \Omega)}s^h \nonumber \\
 &- & \sum_{j=0}^{h-1}\Bigg\vert \Bigg\vert \frac{\partial^j}{\partial \tilde{\nu}^j}(\mu_{a_1} - \mu_{a_2})\Bigg\vert \Bigg\vert_{L^{\infty}(\partial \Omega)}s^j \nonumber \\
 &- &C \vert x - x_0 \vert^{\alpha}s^h, 
\end{eqnarray}

for any $x \in B_{\rho}(z_{\tau}) \cap \Omega$. Observing that we also have

\begin{equation}\label{mu1minusmu2TaylorEstResult3}  
(\mu_{a_1} - \mu_{a_2})(x) \leq \sum_{j=0}^{h-1} \Bigg\vert \Bigg\vert \frac{\partial^j}{\partial \tilde{\nu}^j}(\mu_{a_1} - \mu_{a_2}) \Bigg\vert\Bigg\vert_{L^{\infty}(\partial \Omega)}s^j + Cs^h,
\end{equation}

for any $x \in B_{\rho}(z_{\tau}) \cap \Omega$ and combing together \eqref{mu1minusmu2TaylorEstResult2}, \eqref{mu1minusmu2TaylorEstResult3} and \eqref{Ireland}, leads to

\begin{eqnarray}\label{StabilityFinResI}
& &\vert\vert\Lambda_1 - \Lambda_2 \vert\vert_{*} \vert\vert u_1 \vert\vert_{H^{\frac{1}{2}}(\partial\Omega)} \vert \vert u_2 \vert \vert_{H^{\frac{1}{2}}(\partial\Omega)}\nonumber \\
& & \geq C\Bigg\{ \Bigg\vert\Bigg\vert  \frac{\partial^h}{\partial \tilde{\nu}^h}(\mu_{a_1} - \mu_{a_2})   \Bigg\vert\Bigg\vert_{L^{\infty}(\partial \Omega)} \int_{\Omega \cap B_{\rho}(z_{\tau})} \vert x - z_{\tau}\vert^{2-2(n+m)} d(x, \partial \Omega)^h dx \nonumber \\
& & - \sum_{j=0}^{h-1} \Bigg\vert \Bigg\vert \frac{\partial^j}{\partial \tilde{\nu}^j}(\mu_{a_1} - \mu_{a_2}) \Bigg\vert\Bigg\vert_{L^{\infty}(\partial \Omega)}\int_{\Omega \cap B_{\rho}(z_{\tau})} \vert x  - z_{\tau}\vert^{2-2(n+m)} d(x, \partial \Omega)^j dx \nonumber \\
& & - \int_{\Omega \cap B_{\rho}(z_{\tau})} \vert x - x_0 \vert^{\alpha}\vert x - z_{\tau}\vert^{2-2(n+m)} d(x, \partial \Omega)^h dx \nonumber \\
&& -  \sum_{j=0}^{h-1} \Bigg\vert \Bigg\vert \frac{\partial^j}{\partial \tilde{\nu}^j}(\mu_{a_1} - \mu_{a_2}) \Bigg\vert\Bigg\vert_{L^{\infty}(\partial \Omega)}\int_{\Omega \cap B_{\rho}(z_{\tau})} \vert x - z_{\tau}\vert^{4-2(n+m)} d(x, \partial \Omega)^j dx \nonumber \\
& & -  \int_{\Omega \cap B_{\rho}(z_{\tau})} \vert x - z_{\tau}\vert^{4-2(n+m)} d(x, \partial \Omega)^h dx\Bigg\} - C,
\end{eqnarray}

which leads to

\begin{eqnarray}\label{StabilityFinResII}
& & \Bigg\vert\Bigg\vert  \frac{\partial^h}{\partial \tilde{\nu}^h}(\mu_{a_1} - \mu_{a_2})   \Bigg\vert\Bigg\vert_{L^{\infty}(\partial \Omega)} \int_{\Omega \cap B_{\rho}(z_{\tau})} \vert x - z_{\tau}\vert^{2-2(n+m)} d(x, \partial \Omega)^h dx \leq  \nonumber \\
& & \leq C\Bigg\{ \sum_{j=0}^{h-1} \Bigg\vert \Bigg\vert \frac{\partial^j}{\partial \tilde{\nu}^j}(\mu_{a_1} - \mu_{a_2}) \Bigg\vert\Bigg\vert_{L^{\infty}(\partial \Omega)}\int_{\Omega \cap B_{\rho}(z_{\tau})} \vert x - z_{\tau}\vert^{2-2(n+m)} d(x, \partial \Omega)^j dx \nonumber \\
&+ & \int_{\Omega \cap B_{\rho}(z_{\tau})} \vert x - x_0 \vert^{\alpha}\vert x - z_{\tau}\vert^{2-2(n+m)} d(x, \partial \Omega)^h dx\nonumber \\
&+ &\int_{\Omega \cap B_{\rho}(z_{\tau})} \vert x - z_{\tau}\vert^{4-2(n+m)} d(x, \partial \Omega)^h dx \nonumber \\
&+ &\vert\vert\Lambda_1 - \Lambda_2 \vert\vert_{*} ~ \vert\vert u_1 \vert\vert_{H^{\frac{1}{2}}(\partial\Omega)} ~ \vert \vert u_2 \vert \vert_{H^{\frac{1}{2}}(\partial\Omega)}\Bigg\} + C.
\end{eqnarray}

The first integral on the right hand side of \eqref{StabilityFinResII} can be estimated from above by observing that $\Omega \cap B_{\rho}(z_{\tau}) \subset \{ C \tau \leq \vert x - z_{\tau} \vert \leq 2 \tau_0 \}$ and that $d(x, \partial \Omega) \leq \vert x - z_{\tau} \vert$, hence

\begin{eqnarray}\label{Italy0} 
\int_{\Omega \cap B_{\rho}(z_{\tau})} \vert x - z_{\tau} \vert^{2-2(n+m)} d(x, \partial \Omega)^j \:dx &\leq& \int_{C \tau \leq \vert x - z_{\tau} \vert \leq 2 \tau_0} \vert x - z_{\tau} \vert^{2-2(n+m)} \vert x - z_{\tau} \vert^{j} \:dx \nonumber \\
&=& \int_{C \tau}^{2\tau_0} \sigma^{2-2n-2m+j+n-1} \:d\sigma \int_{\vert \xi \vert = 1} d\Sigma_{\xi} \nonumber \\
&\leq & C \tau^{2-n-2m+j},
\end{eqnarray}

where $d\Sigma_{\xi}$ denotes the surface measure on the unit sphere. Similarly, the other integrals on the right hand side of \eqref{StabilityFinResII} can be estimated from above as

\begin{equation}\label{Italy1}
\int_{\Omega \cap B_{\rho}(z_{\tau})}  \vert x - z_{\tau} \vert^{2-2(n+m)} d(x, \partial \Omega)^h \:dx \leq C \tau^{2-n-2m+h} 
\end{equation}

\begin{equation}\label{Italy2}
\int_{\Omega \cap B_{\rho}(z_{\tau})}  \vert x - x_0 \vert^{\alpha} \vert x - z_{\tau} \vert^{2-2(n+m)} d(x, \partial \Omega)^h \:dx \leq C \tau^{2-n-2m+h +\alpha}, 
\end{equation}

where the integral on the left hand side of \eqref{StabilityFinResII} can be estimated from below (see \cite[p.66]{Sa})

\begin{equation}\label{Italy3}
\int_{\Omega \cap B_{\rho}(z_{\tau})} \vert x - z_{\tau} \vert^{2-2(n+m)} d(x, \partial \Omega)^h \:dx \geq C \tau^{2-n-2m+h}.
\end{equation}

Recalling that by the induction hypothesis we have

\begin{equation}\label{InducHyp}
 \Bigg\vert \Bigg\vert \frac{\partial^j}{\partial \tilde{\nu}^j}(\mu_{a_1} - \mu_{a_2}) \Bigg\vert\Bigg\vert_{L^{\infty}(\partial \Omega)} \leq C \vert\vert \Lambda_1 - \Lambda_2 \vert\vert_{*}^{\delta j},\qquad\textnormal{for\:any}\:j,\quad j\leq h-1.
\end{equation}

By combining \eqref{Italy0} - \eqref{InducHyp} and the $H^{\frac{1}{2}}(\partial\Omega)$-norms of $ u_1, u_2$ (see \cite{A1}, \cite{AG}) we get

\begin{eqnarray}\label{StabilityFinResXIII}
 \Bigg\vert\Bigg\vert  \frac{\partial^h}{\partial \tilde{\nu}^h}(\mu_{a_1} - \mu_{a_2})   \Bigg\vert\Bigg\vert_{L^{\infty}(\partial \Omega)} \tau^{2-n-2m+h}& \leq & C \sum_{j=0}^{h-1} \vert\vert \Lambda_1 - \Lambda_2 \vert\vert^{\delta_j}_* \tau^{2-n-2m+j} \nonumber \\
&+&C \tau^{2-n-2m+\alpha + h} + C + C\tau^{4-n-2m+h} \nonumber \\
&+&\vert\vert \Lambda_1 - \Lambda_2 \vert\vert_{*} \tau^{2-n-2m}. \nonumber \\
&\leq& C \vert\vert \Lambda_1 - \Lambda_2 \vert\vert_{*}^{\delta_{h-1}}  \tau^{2-n-2m} \nonumber \\
&+ &C \tau^{2-n-2m+\alpha +h} + C + C\tau^{4-n-2m+h}, 
\end{eqnarray}

hence

\begin{eqnarray}\label{NearEndResult}
 \Bigg\vert\Bigg\vert  \frac{\partial^h}{\partial \tilde{\nu}^h}(\mu_{a_1} - \mu_{a_2})   \Bigg\vert\Bigg\vert_{L^{\infty}(\partial \Omega)} 
&\leq &C\Big\{ \vert\vert \Lambda_1 - \Lambda_2 \vert\vert_{*}^{\delta_{k-1}}  \tau^{-h} \nonumber \\
 &+ & \tau^{\alpha} + \tau^{n+2m-2-h} + \tau^{2}\Big\}.
\end{eqnarray}

By choosing $m$ sufficiently large, \eqref{NearEndResult} becomes

\begin{equation}\label{second last estimate}
 \Bigg\vert\Bigg\vert  \frac{\partial^h}{\partial \tilde{\nu}^h}(\mu_{a_1} - \mu_{a_2})   \Bigg\vert\Bigg\vert_{L^{\infty}(\partial \Omega)} 
\leq C \Big\{ \vert\vert \Lambda_1 - \Lambda_2 \vert\vert_{*}^{\delta_{k-1}}  \tau^{-h} + \tau^{2}\Big\}.
\end{equation}

Finally, optimising \eqref{second last estimate} with respect to $\tau$ we obtain 

\begin{equation}\label{StabilityFinResult}
 \Bigg\vert\Bigg\vert  \frac{\partial^h}{\partial \tilde{\nu}^h}(\mu_{a_1}- \mu_{a_2})   \Bigg\vert\Bigg\vert_{L^{\infty}(\partial \Omega)} \leq C \vert \vert \Lambda_1 - \Lambda_2 \vert \vert_{*}^{\delta_h},
\end{equation}

which concludes the proof of \eqref{InductProof}.
\end{proof}


\begin{proof}[Proof of Corollary \ref{corollary}]

By an induction argument, for every multi-index $\beta$, $|\beta|\leq h$, we have

\begin{equation}\label{polynomial}
\frac{\partial^{|\beta|}}{\partial x^{\beta}} K (x, \mu_a(x)) = \sum_{|\gamma| + |\delta| \leq |\beta|} P_{\gamma\:\delta} (\mu_a(x),\dots , D^{|\delta|} \mu_a(x)) \: \frac{\partial^{|\gamma|}}{\partial x^{\gamma}}\: \frac{\partial^{|\delta|}}{\partial t} K (x,  t)\Big|_{t=\mu_a(x)}, 
\end{equation}

where $P_{\gamma\:\delta}$ is a polynomial. By hypotheses \eqref{assumption ms smooth}, \eqref{assumption B smooth}, we obtain that $K(x,\mu_{a_i}(x))\in C^{h,\alpha}(\overline\Omega_r)$, for $i=1,2$, hence 

\begin{equation}\label{estimate tensor mua}
||D^h \left(K(x,\mu_{a_1}) - K(x, \mu_{a_2})\right)||_{L^{\infty}(\partial\Omega)}\leq C ||\mu_{a_1} - \mu_{a_2}||^{\alpha}_{C^h(\overline\Omega_r)},
\end{equation}

which, combined with \eqref{stability derivatives}, implies \eqref{stability derivatives K}.

\end{proof}


\end{document}